\documentclass[reqno]{amsart}
\usepackage{amsmath,amssymb}
\usepackage{relsize}
\usepackage{graphicx,color}
\usepackage[all]{xy}
\theoremstyle{plain}
\newtheorem{introtheorem}{Theorem}
\newtheorem{introcorollary}[introtheorem]{Corollary}
\newtheorem{theorem}{Theorem}[section]
\newtheorem{proposition}[theorem]{Proposition}

\newtheorem{corollary}[theorem]{Corollary}

\theoremstyle{definition}
\newtheorem{definition}[theorem]{Definition}
\newtheorem{example}[theorem]{Example}

\theoremstyle{remark}
\newtheorem{remark}[theorem]{Remark}

\def\Q{{\mathbb Q}}

\def\cat0{\mathrm{cat}_0}

\def\dim{\mathrm{dim}}

\begin{document}
	
	\title[]
	{On the Euler-Poincar\'{e} characteristics  of a simply connected  rationally elliptic CW-complex}
	
	\author{Mahmoud Benkhalifa}
	\address{Department of Mathematics. Faculty of  Sciences, University of Sharjah. Sharjah, United Arab Emirates}
	
	\email{mbenkhalifa@sharjah.ac.ae}

	
	
	\begin{abstract} 
		For a simply connected  rationally elliptic CW-complex $X$, we show that the cohomology and the homotopy Euler-Poincaré characteristics are related to two new numerical invariants namely $\eta_{X}$ and $\rho_{X}$ which we define using the Whitehead exact sequences of the Quillen and the Sullivan models of $X$.
	\end{abstract}
	\maketitle
\section{introduction}
A simply connected CW-complex $X$ is called rationally  elliptic if both the graded vector spaces $H^{*}(X; \Q)$ and $\pi_{*}(X) \otimes \Q$ are finite dimensional. To such a space $X$ are attached two numerical invariants namely the cohomology Euler characteristic $\chi_{H}$ and the homotopy Euler characteristic $\chi_{\pi}$ defined by
$$
\chi_{H}=\sum_{i\geq 0}(-1)^{i} \operatorname{dim} H^{i}(X; \Q) \quad, \quad \chi_{\pi}=\sum_{i\geq 0}(-1)^{i} \operatorname{dim} \pi_{i}(X) \otimes \Q.
$$
Computing   $\chi_{H}$ and $\chi_{\pi}$ and studying their properties is a major task in rational homotopy theory. For instance, it is well-known that (see \cite{B4}, Prop. $\left.32.10\right)$
 \begin{itemize}
 	\item $\chi_{H} \geq 0$ and $\chi_{\pi} \leq 0.$
 	\item $\chi_{H}>0$ if and only if $\chi_{\pi}=0.$ 
 \end{itemize}
If $X^{[i]}$ denotes the $i^{\text {th }}$ Postnikov section of $X$ and $X^{i}$ denotes its $i^{\text {th }}$ skeleton, then   we introduce two  new numerical homotopy invariants namely
$$
\rho_{X}=1+\sum_{i \geq 2}(-1)^{i} \operatorname{dim} H^{i}(X^{[i-2]}; \Q) \quad, \quad \eta_{X}=1+\sum_{i \geq 2}(-1)^{i} \operatorname{dim} \Gamma_{i}(X),
$$
where  $\Gamma_{i}(X)=\operatorname{ker}(\pi_{i}(X^{i}) \otimes \Q \longrightarrow \pi_{i}(X^{i}, X^{i-1}) \otimes \Q)$ for $i \geq 2$. 

By exploiting the well-known properties of a rationally  elliptic CW-complex  $X$ as well as the virtue of the Whitehead  exact sequences associated respectively to the Sullivan model and the Quillen model of $X$, we prove the following

\begin{introtheorem} If $X$ is a rationally elliptic CW-complex $X$, then
\begin{enumerate}
	\item $H^{i}(X^{[i-2]}; \Q)\cong \Gamma_{i-1}(X)$ for all $i \geq 2;$
	\item  $\rho_{X}= \eta_{X}>0;$
	\item  	Either $\rho_{X}=\chi_{H}$ or $\rho_{X}=-\chi_{\pi}$;
	\item If $X$ is \text{\rm{2}}--connected, then  $\Gamma_{i}(X)= \pi_{i}(X^{i}) \otimes \Q$ for all $i \geq 2$.                
\end{enumerate}
\end{introtheorem}
As a corollary of theorem 1, we deduce the following result
\begin{introcorollary} Let  $X$ be a rationally elliptic CW-complex $X$. We have the following
	\begin{enumerate}
\item If  $H^{2i}(X^{[2i-2]}; \Q)= 0$ for every $i\geq 2$, then    $X$ is  rationally an odd-dimensional sphere.
\item If  $X$ is \text{\rm{2}}-connected and $ \pi_{2i-1}(X^{2i-2})\otimes\Q=0$ for every $i\geq 1$, then  $X$ is    rationally an odd-dimensional sphere.
\end{enumerate}
\end{introcorollary}	
A special class of rationally  elliptic CW-complexes is formed by the  $F_0$-spaces. Such CW-complex is characterised by the fact that its homotopy Euler characteristics vanishes (see definition \ref{d5}). Thus, from the main theorem of this paper we show
\begin{introcorollary}If $X$ is a \text{\rm{2}}-connected $F_0$-space, then
	$$
	\pi_{2i+1}(X^{2i})\otimes \Q=H^{2i+2}(X^{[2i]}, \Q)=0 \quad, \quad \forall i \geq 1.
	$$
\end{introcorollary}	
 We show our results using standard tools of rational homotopy theory by working algebraically on the models of Quillen and Sullivan of $X$. We refer to \cite{B4} for a general introduction to these techniques. Recall that every simply connected space of finite type has a corresponding differential commutative cochain algebra called the Sullivan model of $X$, unique up to isomorphism, that encodes the rational homotopy of $X$. Dually, every simply connected space $X$ has a differential graded Lie algebra (DGL for short), called the Quillen model of $X$ and unique up to isomorphism, which determines completely the rational homotopy type of $X$.

\medskip 
The paper is organised as follows. In section 2, we recall the definitions of the Whitehead exact sequences associated with the Quillen and the Sullivan models of an rationally elliptic CW-complex as well as we formulate and prove all the results in an algebraic setting. In section 3, we give a mere transcription of the above results in the topological context.
\section{Whitehead exact sequences in rational homotopy theory}
\subsection{Whitehead exact sequence associated with a DGL } Let $W=(W_{\geq 1})$ be a finite dimensional graded vector space over $\Q$ and let $(\mathbb{L}(W), \delta)$ be a DGL.  We define the linear maps
\begin{equation}\label{01}
	j_{i}: H_{i}(\mathbb{L}(W_{\leq i})) \longrightarrow W_{i} \quad, \quad b_{i}: W_{i} \longrightarrow H_{i-1}(\mathbb{L}(W_{\leq i-1})),
\end{equation}
by setting $j_{i}([w+y])=w$ and $b_{i}(w)=[\delta(w)]$, were $[\delta(w)]$ denotes the homology class of $\delta(w)$ in the sub-Lie algebra $\mathbb{L}_{i-1}(W_{\leq i-1}).$ Recall that if $x \in H_{i}(\mathbb{L}(W_{\leq i}))$, then $x=[w+y]$, where $w \in W_{i}, y \in \mathbb{L}_{i}(W_{\leq i-1})$ and $\delta(w+y)=0$.

To every DGL $(\mathbb{L}(W), \delta)$ we can assign (see \cite{B1,B3,B5} for more details) the following long exact sequence
\begin{equation}\label{1}
\cdots \rightarrow W_{i+1} \stackrel{b_{i+1}}{\longrightarrow} \Gamma_{i} \rightarrow H_{i}(\mathbb{L}(W)) \rightarrow W_{i} \stackrel{b_{i}}{\longrightarrow} \cdots
\end{equation}
called the Whitehead exact sequence of $(\mathbb{L}(W), \delta)$, where
\begin{equation}\label{2}
	\Gamma_{i}=\operatorname{ker}(j_{i}: H_{i}(\mathbb{L}(W_{\leq i})) \rightarrow W_{i}) \quad, \quad i \geq 2.
\end{equation}
\begin{remark}\label{r1} First, it is clear that if $H_{i}(\mathbb{L}(W_{\leq i}))=0$, then $\Gamma_{i}=0$.
	
\noindent Next, 	if $(\mathbb{L}(W), \delta)$ satisfies $W_{1}=0$, then
\begin{equation}\label{3}
\Gamma_{i}=H_{i}(\mathbb{L}(W_{\leq i-1}))\,\,\,\,\,\,\,\,\,\,\,\,\,\,\,  i \geq 2.
\end{equation}
Indeed; let us consider the following long exact sequence
$$
\cdots \rightarrow H_{i+1}(\mathbb{L}(W_{\leq i}) / \mathbb{L}(W_{\leq i-1})) \rightarrow H_{i}(\mathbb{L}(W_{\leq i-1})) \rightarrow H_{i}(\mathbb{L}(W_{\leq i})) \rightarrow \cdots
$$
which is associated with the following short exact of chain complexes
$$(\mathbb{L}(W_{\leq i-1}), \delta) \mapsto(\mathbb{L}(W_{\leq i}), \delta) \rightarrow(\mathbb{L}(W_{\leq i}) / \mathbb{L}(W_{\leq i-1}), \widetilde{\delta}),
$$
where $(\mathbb{L}(W_{\leq i}) / \mathbb{L}(W_{\leq i-1}), \widetilde{\delta})$ is the quotient chain complex and the differential $\widetilde{\delta}$ is induced by $\delta .$ Now clearly, we have $(\mathbb{L}(W_{\leq i}) / \mathbb{L}(W_{\leq i-1})) \cong W_{i}$ and as $W_{1}=0$ it follows that $H_{i+1}(\mathbb{L}(W_{\leq i}) / \mathbb{L}(W_{\leq i-1}))=0 .$ As a result we obtain (\ref{3}). 
\end{remark}
\subsection{ Whitehead exact sequence associated with a Sullivan algebra} A Sulivan algebra $(\Lambda V, \partial)$ over $\Q$ is simply connected if $V^{1}=0 .$ Let $(\Lambda(V^{\leq i-2}), \partial)$ be the sub algebra generated by the graded vector apace $V^{\leq i-2} .$ For simplicity let us write
\begin{equation}\label{4}
L^{i}=H^{i}(\Lambda(V^{\leq i-2})) \quad, \quad i \geq 2.
\end{equation}
The Whitehead exact sequence of $(\Lambda V, \partial)$ is defined as follows (see \cite{B1,B2,B4} for more details)
\begin{equation}\label{5}
\cdots \rightarrow H^{i}(\Lambda V) \rightarrow V^{i} \stackrel{b^{i}}{\longrightarrow} L^{i+1} \longrightarrow H^{i+1}(\Lambda V) \longrightarrow V^{i+1} \stackrel{b^{i+1}}{\longrightarrow} \cdots
\end{equation}
Here $b^{i}(v)=[\partial(v)]$, where $[\partial(v)]$ is the cohomology class of the cycle $\partial(v)$ in $\Lambda(V^{\leq i-2})$.

\subsection{Elliptic algebras} A Sullivan algebra $(\Lambda V, \partial)$ is called elliptic if  
$$\operatorname{dim} H^{*}(\Lambda V)=\sum_{i\geq 0}^{} \operatorname{dim}H^{i}(\Lambda V)<\infty\,\,\,\,\,\,\,\,\,\,\text{ and }\,\,\,\,\,\,\,\,\,\,\operatorname{dim} V=\sum_{i\geq 2}^{} \operatorname{dim}V^{i}<\infty.$$
Let us call $n=\max \left\{i: H^{i}(\Lambda V) \neq 0\right\}$ the formal dimension of $(\Lambda V, \partial)$. For an elliptic Sullivan algebra $(\Lambda V, \partial)$, we define the following  two numbers
\begin{eqnarray} 
\chi_{\Lambda V}&=&\sum_{i\geq 0}^{} \operatorname{dim}H^{2i}(\Lambda V)-\sum_{i\geq 0}^{} \operatorname{dim}H^{2i+1}(\Lambda V)=H^{\rm{even}}(\Lambda V)-H^{\rm{odd}}(\Lambda V),\nonumber\\
\chi_{V}&=&\sum_{i\geq 1}^{} \operatorname{dim}V^{2i}-\sum_{i\geq 1}^{} \operatorname{dim}V^{2i+1}=\operatorname{dim}V^{\rm{even}}-\operatorname{dim}V^{\rm{odd}}.\nonumber
\end{eqnarray}
The following are some important properties of elliptic Sullivan algebras.

\begin{theorem}\label{t1} \rm{([4], \S 32)} Suppose $(\Lambda V, \partial)$ is  elliptic of formal dimension $n .$ Then
\begin{enumerate}
\item $\operatorname{dim} V^{\text {odd }} \geq \operatorname{dim} V^{\text {even }};$
\item $V^{i}=0$, for $i \geq 2 n$;
\item $V^{i}=0$, for $i>n$ and $i$ even;
\item There is only one non-trivial $V^{i}$, for $i>n$ and $i$ odd. Necessary, we have $\operatorname{dim} V^{i}=1;$
\item $\chi_{\Lambda V} \geq 0$ and $\chi_{V} \leq 0 .$ Moreover $\chi_{\Lambda V}=0 \Longleftrightarrow \chi_{V}<0.$
\end{enumerate}
\end{theorem}

\begin{proposition}\label{p1} Suppose $(\Lambda V, \partial)$ is elliptic of formal dimension $n.$ 
\begin{enumerate}
\item For every even number i such that $i>n+1$, we have $L^{i}=0$;
\item For every odd number i such that i $>n+1$ we have $L^{i} \cong V^{i-1}$;
\item For every i such that $i> 2n$, we have $L^{i}=0$.
\end{enumerate}
\end{proposition}
\begin{proof} As $(\Lambda V, \partial)$  has formal dimension $n$, it follows that $H^{i}(\Lambda V)=0$ for $i>n$. Therefore, by theorem \ref{t1} the  Whitehead exact sequence of $(\Lambda V, \partial)$ can be written as
$$
 H^{n}(\Lambda V)\rightarrow V^{n} \rightarrow L^{n+1} \rightarrow 0  \rightarrow \cdots\to 0 \rightarrow V^{2n-1}\to L^{2 n} \rightarrow0\rightarrow V^{2n}=0\to L^{2 n+1} \to0
$$
Consequently, the three assertions of proposition \ref{p1} are easily derived.
\end{proof}
\begin{definition}\label{d1} Let $(\Lambda V, \partial)$ be an elliptic Sullivan algebra.  Using (\ref{4}), we define $$\rho_{\Lambda V}=1+\operatorname{dim} L^{\text {even }}-\operatorname{dim} L^{\text {odd }},$$
	 where
	$$\operatorname{dim} L^{\text {even }}=\operatorname{dim} L^4+\operatorname{dim} L^6+\operatorname{dim} L^8+\dots$$$$\operatorname{dim} L^{\text {odd }}=\operatorname{dim} L^5+\operatorname{dim} L^7+\operatorname{dim} L^9+\dots$$
	Note that from (\ref{4}), we deduce that $L^i=0$ for $i=2,3$.
\end{definition}
\begin{proposition}\label{p2}  If $(\Lambda V, \partial)$ is elliptic of formal dimension $n$, then we have
\begin{equation}\label{6}
0\leq \rho_{\Lambda V}-\sum_{i=4}^{n+1}(-1)^{i} \operatorname{dim} L^{i}\leq 2.
\end{equation}
\end{proposition}
\begin{proof} First by (\ref{3}) of proposition \ref{p1}, we have
$$
\rho_{\Lambda V}=1+\sum_{i=4}^{2 n}(-1)^{i} \operatorname{dim} L^{i}=1+\sum_{i=4}^{n+1}(-1)^{i} \operatorname{dim} L^{i}+\sum_{i=n+2}^{2 n}(-1)^{i} \operatorname{dim} L^{i}.
$$
Next, let us focus on the integer $\sum_{i=n+2}^{2 n}(-1)^{i} \operatorname{dim} L^{i} .$ Using  the relations (3) and (4) of theorem  \ref{t1} and  (2)  of proposition \ref{p1}, we deduce that 
$$-1\leq\sum_{i=n+2}^{2 n}(-1)^{i} \operatorname{dim} L^{i}\leq 1.$$
Therefore 
$$
\sum_{i=4}^{n+1}(-1)^{i} \operatorname{dim} L^{i}\leq \rho_{\Lambda V}\leq \sum_{i=4}^{n+1}(-1)^{i} \operatorname{dim} L^{i}+2,
$$
as wanted.
\end{proof}

The next result shows the relationship between the $\rho_{\Lambda V}, \chi_{\Lambda V}$ and $\chi_{V}.$

\begin{theorem}  If $(\Lambda V, \partial)$ is an elliptic Sullivan algebra, then  $\rho_{\Lambda V}=\chi_{\Lambda V}-\chi_{V}$.
\end{theorem} 
\begin{proof} Let us consider the Whitehead exact sequence of $(\Lambda V, \partial)$ given in (\ref{5})
$$
\cdots \rightarrow H^{i}(\Lambda V) \rightarrow V^{i} \stackrel{b^{i}}{\longrightarrow} L^{i+1} \longrightarrow H^{i+1}(\Lambda V) \longrightarrow V^{i+1} \stackrel{b^{i+1}}{\longrightarrow} \cdots
$$
Let us write $H^{*}=H^{*}(\Lambda V)$.  For every $i \geq 2$, let
$$
V^{i, 1}=\operatorname{ker} b^{i} \quad, \quad V^{i, 2}=\operatorname{im} b^{i} \quad, \quad L^{i+1,1}=\operatorname{im} i^{i+1}.
$$
The exactness of the above sequence implies that
\begin{equation}\label{7}
V^{i} \cong V^{i, 1} \oplus V^{i, 2} \quad,\quad L^{i+1} \cong V^{i, 2} \oplus L^{i+1,1}\quad, \quad H^{i+1} \cong L^{i+1,1} \oplus V^{i+1,1}.
\end{equation}
Using the direct summands in (\ref{7}) we deduce the following
$$
\begin{aligned}
	\operatorname{dim} V^{3} &=\operatorname{dim} H^{3}+\operatorname{dim} V^{3,2} \\
	&=\operatorname{dim} H^{3}+\operatorname{dim} L^{4}-\operatorname{dim} L^{4,1} \\
	&=\operatorname{dim} H^{3}+\operatorname{dim} L^{4}-\operatorname{dim} H^{4}+\operatorname{dim} V^{4,1} \\
	&=\operatorname{dim} H^{3}+\operatorname{dim} L^{4}-\operatorname{dim} H^{4}+\operatorname{dim} V^{4}-\operatorname{dim} V^{4,2}.
\end{aligned}
$$
Iterating the above process and  taking into consideration that $(\Lambda V, \partial)$ is elliptic which implies that  this process must stop, we get  the following formula
$$
\operatorname{dim} V^{3}=-\sum_{i \geq 3}(-1)^{i} \operatorname{dim} H^{i}+\sum_{i \geq 4}(-1)^{i} \operatorname{dim} L^{i}+\sum_{i \geq 4}(-1)^{i} \operatorname{dim} V^{i},
$$
which implies
\begin{equation}\label{8}
\sum_{i \geq 3}(-1)^{i} \operatorname{dim} H^{i}-\sum_{i \geq 3}(-1)^{i} \operatorname{dim} V^{i}=\sum_{i \geq 4}(-1)^{i} \operatorname{dim} L^{i}.
\end{equation}
As $(\Lambda V, \partial)$ is simply connected, it follows that $ H^{0}\cong \Q$, $H^{1}=0$ and $H^{2} \cong V^{2}$. Therefore the relation (\ref{8}) becomes
\begin{equation*}\label{9}
	(\operatorname{dim} H^{\text {even }}-\operatorname{dim} H^{\text {odd }})-(\operatorname{dim} V^{\text {even }}-\operatorname{dim} V^{\text {odd }})=1+\sum_{i \geq 4}(-1)^{i} \operatorname{dim} L^{i},
\end{equation*}
implying $\rho_{\Lambda V}=\chi_{\Lambda V}-\chi_{V}$.
\end{proof}

\begin{corollary}\label{c1} If $(\Lambda V, \partial)$ is an elliptic Sullivan algebra, then $\rho_{\Lambda V}>0 .$ Moreover if $\chi_{\Lambda V}>0$, then $\rho_{\Lambda V}=\chi_{\Lambda V}$.
\end{corollary}
\begin{proof} As $\rho_{\Lambda V}=\chi_{\Lambda V}-\chi_{V}$, the relation (5) of theorem \ref{t1} implies that $\rho_{\Lambda V} \geq 0 .$ Now if we assume that $\rho_{\Lambda V}=0$, then we derive that $\chi_{\Lambda V}=\chi_{V}$. Again by the relation (5) of theorem \ref{t1}, on one hand we have $\chi_{\Lambda V}=\chi_{V}=0$ and on the other hand we have $\chi_{V}<0$ which is impossible. Thus, $\rho_{\Lambda V}>0$.

\noindent Now if $\chi_{\Lambda V}>0$, then the relation (4) of theorem \ref{t1} implies that $\chi_{V}=0 .$ Hence $\rho_{\Lambda V}=\chi_{\Lambda V}-\chi_{V}=\chi_{\Lambda V}$.
\end{proof}
\begin{corollary}\label{c5} If $(\Lambda V, \partial)$ is an  elliptic Sullivan algebra, then $L^{\rm{even}}\geq L^{\rm{odd}}$. Moreover, if $L^{\rm{even}}= 0$, then  $\operatorname{dim} H^{*}(\Lambda V)=\operatorname{dim}V^{*} +1$.
\end{corollary}
\begin{proof} By combining  definition \ref{d1} and corollary \ref{c1}, we obtain 
	$$\rho_{\Lambda V}=1+\operatorname{dim} L^{\text {even }}-\operatorname{dim} L^{\text {odd }}>0\Longrightarrow L^{\rm{even}}\geq L^{\rm{odd}}.$$ 
 Thus,  if $L^{\rm{even}}= 0$, then $L^{\rm{odd}}= 0$ implying that $L^{i}= 0\,,\,\forall i\geq 2$. Finally, from the Whitehead exact sequence (\ref{5}) of $(\Lambda V, \partial)$ , we deduce that
  \begin{equation}\label{99}
  H^{i}(\Lambda V)\cong V^{i} \,\,\,\,\,\,\,\,\,,\,\,\,\,\,\,\,\,\,\,\,\forall i\geq 2. 
  \end{equation}
  Consequently, if $n$ is the formal dimension of $(\Lambda V, \partial)$, then $ V^{i}=0 \,,\,\forall i> n$ and we obtain
 \begin{eqnarray}
 \operatorname{dim} H^{*}(\Lambda V)\hspace{-2mm}&=&\hspace{-2mm}\operatorname{dim} H^{0}(\Lambda V)+\operatorname{dim} H^{1}(\Lambda V)+\operatorname{dim} H^{2}(\Lambda V)+\cdots+\operatorname{dim} H^{n}(\Lambda V)\nonumber\\
 \hspace{-2mm}&=&\hspace{-2mm} 1+0+\operatorname{dim} V^{2}+\operatorname{dim} V^{3}+\cdots+\operatorname{dim} V^{n}=1+\operatorname{dim} V^{*}\nonumber 
  \end{eqnarray}
as desired.
\end{proof}
\section{Topological applications}
Let $X$ be a simply connected CW-complex. 
For  every $i \geq 2$, we define the vector space
\begin{equation}\label{10}
	\Gamma_{i}(X)=\operatorname{ker}(\pi_{i}(X^{i}) \otimes \Q \longrightarrow \pi_{i}(X^{i}, X^{i-1}) \otimes \Q).
\end{equation}
Here $X^{i}$ denotes the $i^{\text {th }}$-skeleton of $X$.

Thus, if $(\mathbb{L}(W), \delta)$ denotes the Quillen model of $X$, then by virtue of the properties of this model,  we obtain  the following identifications (valid for any $i \geq 2$)
\begin{equation}\label{11}
\pi_{i}(X) \otimes \Q \cong H_{i-1}(\mathbb{L}(W)) \quad,\quad H_{i}(X; \Q) \cong W_{i-1} \quad, \quad \Gamma_{i}(X) \cong \Gamma_{i-1},
\end{equation}
where $\Gamma_{i}$ is defined in (\ref{2}). Therefore the Whitehead exact sequence (\ref{1}) of this model can be written as
\begin{equation}\label{12}
\cdots \longrightarrow H_{i+1}(X; \Q) \stackrel{b_{i+1}}{\longrightarrow} \Gamma_{i}(X) \stackrel{}{\longrightarrow} \pi_{i}(X) \otimes \Q \stackrel{h_{i}}{\longrightarrow}  H_{i}(X; \Q) \stackrel{b_{i}}{\longrightarrow} \cdots
\end{equation}
	where   $h_i$ is the Hurewicz homomorphism. 
\begin{remark}\label{r2} Firstly, since $(\mathbb{L}(W_ {\leq i-1}), \delta)$ can be chosen as the Quillen model of the skeleton $X^{i}$ (see \cite{B4}, pp. 323) implying that
	$$\pi_{i}(X^{i})\otimes \Q\cong  H_{i-1}(\mathbb{L}(W_{\leq i-1}))\,\,\,\,\,\,\,\,\,\,\,\,,\,\,\,\,\,\,\,\,\,\forall i\geq 2$$
and based on remark \ref{r1} and the identifications (\ref{11}), it follows that 
\begin{equation}\label{22}
\pi_{i}(X^{i}) \otimes \Q=0\Longrightarrow\Gamma_{i}(X)=0\,\,\,\,\,\,\,\,\,\,\,\,,\,\,\,\,\,\,\,\,\,\forall i\geq 2
	\end{equation}
	Secondly,	if the space $X$ is \text{\rm{2}}-connected, then by remark \ref{r1} we deduce that 
\begin{equation*}\label{13}
	\Gamma_{i} \cong H_{i}(\mathbb{L}(W _ {\leq i-1})) \cong \pi_{i+1}(X^{i}) \otimes \Q \quad, \quad i \geq 2,
\end{equation*}
and the Whitehead exact sequence (\ref{12}) becomes
$$
\cdots \longrightarrow H_{i+1}(X; \Q) \stackrel{b_{i+1}}{\longrightarrow} \pi_{i}(X^{i-1}) \otimes \Q \stackrel{i_{i}}{\longrightarrow} \pi_{i}(X) \otimes \Q \longrightarrow H_{i}(X; \Q) \stackrel{b_{i}}{\longrightarrow} \cdots
$$
\end{remark}
Dually, if $(\Lambda V, \partial)$ is the Sullivan model of $X$, then we have
\begin{equation}\label{14}
H^{i}(X; \Q) \cong H^{i}(\Lambda V)\,\,\,\,\,\,\,\,\,\,\,\,\,\,\,\,,\,\,\,\,\,\,\,\,\,\,\, H^{i+1}(X^{[i-1]}; \Q) \cong H^{i+1}(\Lambda V^{\leq i-1}),$$
$$ V^{i} \cong \operatorname{Hom}(\pi_{i}(X) \otimes \Q; \Q),
\end{equation}
where $X^{[i]}$ denotes the $i^{\text {th }}$ Postikov section of $X .$ Therefore,  the Whitehead exact sequence of this model can be written as
\begin{equation}\label{15}
\cdots \rightarrow \operatorname{Hom}(\pi_{i}(X) \otimes \Q, \Q) \overset{b^{i}}{\rightarrow} H^{i+1}(X^{[i-1]}; \Q) \rightarrow H^{i+1}(X; \Q) \rightarrow \operatorname{Hom}(\pi_{i+1}(X) \otimes \Q, \Q) \rightarrow \cdots
\end{equation}
Recall that  a  simply connected CW-complex $X$ is  rationally  elliptic if 
	$$\operatorname{dim} H^{*}(X;\Q)=\sum_{i\geq 0}^{} \operatorname{dim}H^{i}(X;\Q)<\infty\,\,\,\text{ and }\,\,\,\,\operatorname{dim} \pi_{*}(X)\otimes \Q=\sum_{i\geq 2}^{} \operatorname{dim} \pi_{i}(X)\otimes \Q<\infty$$
Note that,  by virtue of the Sullivan model, $X$ is   rationally  elliptic if and only if its Sullivan model is elliptic. In this case the formal dimension of $X$ is the formal dimension of its Sullivan model.
\begin{definition}\label{d3} Let  $X$ be a rationally  elliptic CW-complex.  We define
$$\eta_{X}=1+\dim\,\Gamma_{\rm {even }}(X)-\dim \,\Gamma_{\rm{odd }}(X)$$
where $\Gamma_{*}(X)$ is defined in (\ref{10})
\end{definition}
\begin{definition}\label{d4} If $X$ is a  rationally  elliptic CW-complex and $(\Lambda V, \partial)$ its Sullivan model, then we define $\rho_{X}=\rho_{\Lambda V}$, where $\rho_{\Lambda V}$ is given  in definition \ref{d1}.
\end{definition}
Subsequently,  we need the following proposition showing that  $H^{i+1}(X^{[i-1]}, \Q)$ and $\Gamma_{i}(X)$ are isomorphic,  as vector spaces, for every $i\geq 2$. 
\begin{proposition}\label{p3} If $X$ is a rationally  elliptic CW-complex, then
\begin{equation}\label{16}
H^{i+1}(X^{[i-1]}; \Q) \cong \text{\rm{Hom}}(\Gamma_{i}(X), \Q), \quad i \geq 2.
\end{equation}
\end{proposition}
\begin{proof} Applying the exact functor $\text{\rm{Hom}}(.,\Q)$ to the exact sequence (\ref{12}) we obtain
\begin{equation}\label{17}
\cdots \leftarrow H^{i+1}(X; \Q) \leftarrow \text{\rm{Hom}}(\Gamma_{i}(X), \Q) \leftarrow \text{\rm{Hom}}(\pi_{i}(X) \otimes \Q, \Q) \leftarrow H^{i}(X; \Q) \leftarrow \cdots
\end{equation}
Taking in account that

\begin{itemize}
	\item All the vector spaces involved are finite dimensional.
	\item The two maps $H^{i}(X; \Q) \rightarrow\text{\rm{Hom}}(\pi_{i}(X) \otimes \Q, \Q)$ appearing in (\ref{15}) and (\ref{17}) are the same morphism for all $i \geq 2$.
		\item  $\text{\rm{Hom}}(H_{*}(X; \Q); \Q)\cong H^{*}(X; \Q)$.
\end{itemize}
and by comparing the sequences $(\ref{15}),(\ref{17})$, we get $(\ref{16})$.
\end{proof}
\begin{remark}\label{r3} By virtue of the properties of the Sullivan model   $(\Lambda V, \partial)$ of the space $X$, it is important to mention that for  all $i \geq 2$, the linear map 
	$$ b^{i}: V^{i} \longrightarrow H^{i+1}(\Lambda(V ^{\leq i-1})),$$ induced by the differential $\partial$, is the dual of the $i$-invariant
$$
k_{i} \in H^{i+1}(X^{[i-1]}; \pi_{i}(X))=\text{\rm{Hom}}(H_{i+1}(X^{[i-1]}), \pi_{i}(X)),
$$
which is the dual of the map $\Gamma_{i}(X) \overset{i_{i}}{\longrightarrow} \pi_{i}(X) \otimes \Q$ given in (\ref{12}).
\end{remark}
A mere transcription into the topological context of the above results obtained in section 2 provides the following applications.
\begin{corollary}\label{c2} If $X$ is a rationally  elliptic CW-complex of formal dimension $n$, then
\begin{enumerate}
\item For every even number i such that $i>n+1$, we have
$$
H^{i}(X^{[i-2]}, \Q)=\Gamma_{i-1}(X)=0.$$
\item For every i such that $i>2 n$, we have $H^{i}(X^{[i-2]}, \Q)=\Gamma_{i-1}(X)=0$.
\item  If $X$ is \text{\rm{2}}-connected, then  $\Gamma_{i}(X)\cong \pi_{i}(X^{i-1})\otimes\Q $ for every $i\geq 2$.
\end{enumerate}
\end{corollary}
\begin{proof}  It follows from remark \ref{r1},    propositions \ref{p1} and \ref{p3}.
\end{proof}

The next result establishes that the two numerical invariants $\eta_{X}$ and $\rho_{Y}$ are equal although they are defined differently.
\begin{theorem}\label{t3} 
	If $X$ is  a rationally  elliptic CW-complex, then $\eta_{X}=\rho_{X}$.
\end{theorem}	
	\begin{proof} The result follows from the definitions \ref{d1} and  \ref{d3} and by applying proposition \ref{p3} after  taking into account the identifications (\ref{11}).
\end{proof}
\begin{corollary}\label{c4} If $X$ is  a rationally  elliptic CW-complex of formal dimension $n$, then
	$$
	0 \leq \rho_{X}-\sum_{i=4}^{n+1}(-1)^{i} \dim\,H^{i}(X^{[i-2]}; \Q) \leq 2 \,\,\,\,\,\,\,\,,\,\,\,\,\,\,	0\leq \eta_{X}-\sum_{i=3}^{n}(-1)^{i} \dim\, \Gamma_{i}(X)  \leq 2$$
\end{corollary}
\begin{proof} It follows from propositions \ref{p2}, \ref{p3} and theorem \ref{t3}.
\end{proof}
\begin{proposition}\label{p5} 
	If $X$ is  a rationally  elliptic CW-complex  such that 
	$$H^{2i}(X^{[2i-2]}; \Q)=0\,\,\,\,\,\,\,\,,\,\,\,\,\,\,	 \forall i\geq 2,$$  then we have
	\begin{enumerate}
		\item $H^{2i+1}(X^{[2i-1]}; \Q)=0$  for every $i\geq 2$;
		\item $\Gamma_{i}(X)=0$  for every $i\geq 2$;
		\item $\operatorname{dim} H^{*}(X;\Q)= \operatorname{dim} \pi_{i}(X)\otimes\Q +1$.
		\end{enumerate}
	\end{proposition}	
\begin{proof} All the assertions are a consequence of corollary \ref{c5} and the formula (\ref{16}). 
\end{proof}
\begin{theorem}\label{cc1} Let  $X$ be a rationally elliptic CW-complex $X$. If  $H^{2i}(X^{[2i-2]}; \Q)= 0$ for every $i\geq 2$, then    $X$ is  rationally an odd-dimensional sphere.
\end{theorem}
\begin{proof}  
Let $(\mathbb{L}(W), \delta)$ denote the Quillen model of $X$. First, if  $H^{2i}(X^{[2i-2]}; \Q)= 0$ for every $i\geq 2$, then by (\ref{4}) and the identification (\ref{14}), we deduce that $L^{\text {even }}=0$ implying $L^{\text {odd }}=0$ according to  corollary \ref{c4}.  Next, by proposition \ref{16},  we get $\Gamma_{i}=0$ for every $i\geq 2$, where $\Gamma_{i}=0$  is defined in (\ref{2}). Therefore, the linear map  $b_{i+1}:W_{i+1} \to \Gamma_{i}$,  given in (\ref{1}), is trivial for $i\geq 2$. Now using the relation  (\ref{01}), it follows that $b_{i}(w)=[\delta(w)]$,  for $i\geq 2$ and  every $w\in W_{i+1}$.

Consequently, using the properties of the  Quillen model (see \cite{B4}, pp. 323), it follows that all the attaching maps of the space $X$ are rationally trivial.  Consequently the  Quillen model  of $X$ can be chosen as    $(\mathbb{L}(W), 0)$   (with trivial differential). Moreover,  and as the space $X$ is rationally   elliptic, we must have  $W =W_{2k}\cong\Q$ for a certain $k\geq 1$ . Thus,  the Quillen model of $X$ has the form  $(\mathbb{L}(W_{2i}), 0)$ implying that $X$ is rationally  an odd-dimensional sphere.
\end{proof}
\begin{corollary} Let  $X$ be a rationally elliptic CW-complex $X$.  If $ \pi_{2i-1}(X^{2i-1})\otimes\Q=0 $ for every $i\geq 1$, then  $X$ is    rationally an odd-dimensional sphere.
\end{corollary}
\begin{proof}  
	If    $ \pi_{2i-1}(X^{2i-1})\otimes\Q=0 \,,\,\forall i\geq 2$, then by the implication (\ref{22}) and corollary \ref{c2} we deduce that  $H^{2i}(X^{[2i-2]}; \Q)= 0\,,\,\forall i\geq 2$. Then we apply theorem  \ref{cc1}
\end{proof}
\begin{corollary} Let  $X$ be a rationally elliptic CW-complex $X$.  If  $X$ is \text{\rm{2}}-connected and $ \pi_{2i-1}(X^{2i-2})\otimes\Q=0$ for every $i\geq 1$, then  $X$ is    rationally an odd-dimensional sphere
\end{corollary}
\begin{proof}  
	If   $X$ is \text{\rm{2}}-connected and $ \pi_{2i-1}(X^{2i-2})\otimes\Q=0$ for every $i\geq 1$, then by   corollary  \ref{c2}   we deduce that  $H^{2i}(X^{[2i-2]}; \Q)= 0\,,\,\forall i\geq 2$. Then we apply theorem  \ref{cc1}.
	\end{proof}
\begin{example}\label{e1} The purpose of this example is to compute the numerical invariants $\rho_{\mathbb{C} P_{0}^{n}}$ and $\eta_{\mathbb{C} P_{0}^{n}}$, for the space $\mathbb{C} P_{0}^{n}$ which is the rationalised  complex projective space $\mathbb{C} P^{n}$. Indeed; the Sullivan model of $\mathbb{C} P_{0}^{n}$ is $(\Lambda(V^{2} \oplus V^{2 n+1}), \partial)$ with $V^{2}=$ $\langle x\rangle, V^{2 n+1}=\langle y\rangle$ and $\partial y=x^{n+1}$ (see \cite{B4}, example 5, page 333). As we have $L^{i}=$ $H^{i}(\Lambda(V^{\leq i-2}))$, it follows that $L^{\geq 2 n+3}=0 .$ Moreover it is easy to see that $\operatorname{dim} L^{i}=1$ for $i$ even and $\operatorname{dim} L^{i}=0$ for $i$ odd. Therefore $\rho_{\mathbb{C P}_{0}^{n}}=1+L^{\text {even }}=1+n$ and since $\chi_{\pi}(\mathbb{C} P_{0}^{n})=\operatorname{dim} V^{2}-\operatorname{dim} V^{2 n+1}=0$, we deduce that $\rho_{\mathbb{C} P^{n}}=\chi_{\mathbb{C P}_{0}^{n}}=1+n.$ As a result we get $\eta_{\mathbb{C} P^{n}}=1+n$. Notice that the Quillen model of $\mathbb{C} P_{0}^{n}$ is given by
$$
\mathbb{L}(W_{1}, \ldots, W_{2 n-1}) \quad, \quad W_{i}=\left\langle w_{i}\right\rangle \quad, \quad \delta(w_{k})=\frac{1}{2} \sum_{i+j=k}\left[w_{i}, w_{j}\right]
.$$
Hence, for $\Gamma_{i}=H_{i}(\mathbb{L}(W_{\leq i-2}))$, we deduce that $\eta_{\mathbb{C P}_{0}^{n}}=1+\Gamma_{\text {even }}-\Gamma_{\text {odd }}=1+n$.
\end{example}
\subsection{$F_0$-spaces}
\begin{definition}\label{d5} An $F_0$-space is a  rationally  elliptic CW-complex $X$  whose rational cohomology
algebra is given by	$\Q[x_{1},\dots,x_{n}]/(P_{1},\dots,P_{n})$,  where the polynomials $ P_{1},\dots,P_{n} $ form a regular sequence in $\Q[x_{1},\dots,x_{n}]$, i.e., $ P_{1}\neq 0$ and for every $i\geq 1$,  $P_{i}$ is not a zero divisor in  $\Q[x_{1},\dots,x_{n}]/(P_{1},\dots,P_{i-1})$. 
\end{definition}
For instance, products of even spheres,  complex  Grassmannian manifolds  and homogeneous spaces $G/H$ such that  rank $G$ = rank $H$ are $F_0$-spaces.
\begin{proposition}\label{t13}(\cite{B4} Proposition  \rm{32.10}).
Let  $X$  be a rationally  elliptic CW-complex and $(\Lambda V, \partial)$ its Sullivan model. The following   statements are equivalent 
	\begin{enumerate}
		\item $X$  is  an $F_0$-space;
			\item  $H^{\rm{odd}}(\Lambda V)=0;$
		\item $\operatorname{dim}V^{\rm{even}}-\operatorname{dim}V^{\rm{odd}}=0$ and $(\Lambda V, \partial)$ is  pure, i.e., $\partial(V^{\text {even }})=0$ and $\partial(V^{\text {odd }}) \subseteq \Lambda V^{\text {even }}.$
	\end{enumerate}
\end{proposition}
\begin{theorem}\label{t4} If $X$ is an $F_0$-space, then
$$
\Gamma_{2i+1}(X)=H^{2i+2}(X^{[2i]}, \Q)=0 \quad, \quad \forall i \geq 1.
$$
Moreover, if  $X$ is \text{\rm{2}}-connected, then $\pi_{2i+1}(X^{2i})\otimes \Q=0$ for every   $i \geq 1$.
\end{theorem}
\begin{proof}  By (\ref{5}), the Whitehead exact sequence of $(\Lambda V, \partial)$ can be written as
\begin{equation*}\label{66}
\cdots \rightarrow H^{2 i}(\Lambda V) \rightarrow V^{2 i} \stackrel{b^{2 i}}{\longrightarrow} L^{2 i+1} \longrightarrow H^{2 i+1}(\Lambda V) \longrightarrow V^{2 i+1} \stackrel{b^{2 i+1}}{\longrightarrow} \cdots
\end{equation*}
As $X$ is an $F_0$-space, then by proposition \ref{t13}, the Sullivan model $(\Lambda V, \partial)$ of $X$ satisfies  $H^{\text {odd }}(\Lambda V)=0$ and  $\partial(V^{\text {even }})=0$, it follows that the maps $b^{\text {even }}=0 .$ Consequently,  $L^{\text {odd }}=0.$ So,  taking into account the identifications (\ref{11}) and (\ref{14}), the result follows from the relation (\ref{16}).
\end{proof}
\section*{Acknowledgments}
The author is deeply grateful to the referee for a careful reading of the paper
and valuable suggestions that greatly improved the manuscript.

\bibliographystyle{amsplain}

\end{document}